\documentclass{article}

\usepackage{microtype}
\usepackage{graphicx}
\usepackage{subfigure}
\usepackage{booktabs} 
\usepackage{hyperref}
\usepackage{amsfonts,amssymb,amsmath,amsthm,appendix}

\newcommand{\Eset}{\mathbb{E}}

\newcommand{\Nset}{\mathbb{N}}

\newcommand{\Rset}{\mathbb{R}}

\newcommand{\Bcal}{{\cal B}}

\newcommand{\Ecal}{{\cal E}}
\newcommand{\Fcal}{{\cal F}}
\newcommand{\Gcal}{{\cal G}}

\newcommand{\Ncal}{{\cal N}}

\newcommand{\Pcal}{{\cal P}}

\newcommand{\Rcal}{{\cal R}}
\newcommand{\Scal}{{\cal S}}
\newcommand{\Tcal}{{\cal T}}
\newcommand{\Ucal}{{\cal U}}
\newcommand{\Vcal}{{\cal V}}

\newcommand{\Xcal}{{\cal X}}

\newcommand{\Abf}{{\bf A}}

\newcommand{\Ebf}{{\bf E}}

\newcommand{\Hbf}{{\bf H}}
\newcommand{\Ibf}{{\bf I}}

\newcommand{\Mbf}{{\bf M}}

\newcommand{\Qbf}{{\bf Q}}

\newcommand{\Wbf}{{\bf W}}

\newcommand{\Ybf}{{\bf Y}}

\newcommand{\btheta}{{\bar{\theta}}}

\newcommand{\1}{{\mathbf{1}}}

\newcommand{\ebar}{{\bar{e}}}

\newtheorem{lem}{Lemma}
\newtheorem{thm}{Theorem}

\newtheorem{assump}{Assumption}
\theoremstyle{remark}
\newtheorem{remark}{Remark}

\usepackage{enumerate}
\usepackage[shortlabels]{enumitem}

\usepackage[accepted]{icml2019}

\usepackage{color}

\begin{document}

\twocolumn[
\icmltitle{Finite-Time Analysis of Distributed TD(0) with Linear Function Approximation for Multi-Agent Reinforcement Learning}

\begin{icmlauthorlist}
\icmlauthor{Thinh T. Doan}{to,goo}
\icmlauthor{Siva Theja Maguluri}{to}
\icmlauthor{Justin Romberg}{goo}
\end{icmlauthorlist}

\icmlaffiliation{to}{School of Industrial and Systems Engineering}
\icmlaffiliation{goo}{School of Electrical and Computer Engineering, Georgia Instiute of Technology, GA, 30332, USA}

\icmlcorrespondingauthor{Thinh T. Doan}{thinhdoan@gatech.edu}

\icmlkeywords{Machine Learning, ICML}
\vskip0.3in]

\printAffiliationsAndNotice{}

\begin{abstract}
We study the policy evaluation problem in multi-agent reinforcement learning. In this problem, a group of agents work cooperatively to evaluate the value function for the global discounted accumulative reward problem, which is composed of local rewards observed by the agents. Over a series of time steps, the agents act, get rewarded, update their local estimate of the value function, then communicate with their neighbors.  The local update at each agent can be interpreted as a distributed consensus-based variant of the popular temporal difference learning algorithm TD(0). While distributed reinforcement learning algorithms have been presented in the literature, almost nothing is known about their convergence rate.  Our main contribution is providing a finite-time analysis for the convergence of the distributed TD(0) algorithm. We do this when the communication network between the agents is time-varying in general. We obtain an explicit upper bound on the rate of convergence of this algorithm as a function of the network topology and the discount factor. Our results mirror what we would expect from using distributed stochastic gradient descent for solving convex optimization problems.  
\end{abstract}


\section{Introduction}\label{sec:intro}
Reinforcement learning {(\sf RL)} offers  a general paradigm for learning optimal policies in stochastic control problems based on simulation \cite{Sutton1998_book,Bertsekas1999_book,Szepesvari2010_book}. In this context, an agent seeks to find an optimal policy through interacting with the environment, often modeled as a Markov Decision Process {(\sf MDP)}, with the goal of optimizing its long-term future reward (or cost). During the last few years, {\sf RL} has been recognized as a crucial solution for solving many challenging practical problems, such as, autonomous driving \cite{Chen2015_DeepDrive}, robotics \cite{Gu2017_DeepRL}, helicopter flight \cite{Abbeel2007_NIPS},  board games \cite{Silver2016_MasteringTG}, and power networks \cite{Kar2013_QDLearning}.

A central problem in {\sf RL} is to estimate the accumulative reward (value function) for a given stationary policy of an {\sf MDP}, often referred to as the policy evaluation problem. This problem arises as a subproblem in the important policy iteration method in {\sf RL} \cite{Sutton1998_book,Bertsekas1999_book}. Perhaps, the most popular method for solving this  problem is temporal-difference learning ({\sf TD}$(\lambda)$), originally proposed by Sutton \cite{Sutton1988_TD} and analyzed explicitly for various scenarios in \cite{Dayan1992_TD,Gurvits1994_TD,Pineda1997_TD,Tsitsiklis1997_TD, Tsitsiklis1999_TDAverage}. This method approximates the long-term future cost as a function of current state, and depends on a scalar $\lambda\in[0,1]$ that controls a trade-off between the accuracy of the approximation and the susceptibility to simulation noise. In this paper, we focus on the special case $\lambda = 0$, i.e., the {\sf TD}$(0)$ algorithm with function approximation, which has received a great success for solving many complicated problems involving a very large state space \cite{Tesauro1995_TD0,Mnih2015_Nature,Silver2016_MasteringTG}.  This algorithm has a straightforward implementation, and can be executed incrementally as observation are made. 

Our interest in this paper is to study the policy evaluation problem in multi-agent reinforcement learning ({\sf MARL}), where a group of agents operate in an environment. We are motivated by broad applications of the multi-agent paradigm within engineering, for example, mobile sensor networks \cite{CortesMKB2004,Ogren2004_TAC}, cell networks \cite{Bennis2013_CellNetworks}, and power networks \cite{Kar2013_QDLearning}. In this context, each agent takes its own action based on the current state, and consequently a new state is determined. Moreover, the agents receive different local rewards, which are the functions of their current state, their new state, and their action. We assume that each agent only knows its own local reward. Their goal is to cooperatively evaluate the global accumulative reward based only on their local interactions. For solving this problem, our focus is to consider a distributed variant of {\sf TD}$(0)$ algorithm with linear function approximation, where our goal is to provide a finite-time analysis of such distributed {\sf TD}$(0)$ in the context of {\sf MARL}. To the best of our knowledge, such finite-time analysis for distributed {\sf TD}$(0)$ is not available in the existing literature.

\subsection{Existing Literature}
Despite its simple implementation, theoretical analysis for the performance of {\sf TD} is quite complicated. In general, {\sf TD} method is not the true stochastic gradient descent ({\sf SGD}) for solving any static optimization problems,  making it challenging to characterize the consistency and quantify the progress of this method. A dominant approach to study the asymptotic convergence of {\sf TD} learning is to use tools from stochastic approximation ({\sf SA}), specifically, the ordinary differential equation ({\sf ODE}) method. In particular, Tsitsiklis and Van Roy considered a policy evaluation problem on a discounted {\sf MDP} for both finite and infinite state spaces with linear function approximation \cite{Tsitsiklis1997_TD}.  By viewing {\sf TD} as a stochastic approximation for solving a suitable Bellman equation, they characterized the almost sure convergence of this method based on the {\sf ODE} approach. That is, under the right conditions, the {\sf SA} update asymptotically follows the trajectory of a stable {\sf ODE}. The convergence of {\sf SA} is then equivalent to the convergence of this {\sf ODE} solution, which can be shown by using Lyapunov theorem in control theory. Following this work, Borkar and Meyn provided a general and unified framework for the convergence of {\sf SA} with broad applications in {\sf RL} \cite{Borkar2000_ODE}. More general results in this area can be found in the monograph by Borkar \cite{borkar2008}. 

While the asymptotic convergence of {\sf TD} algorithms is well-known, very little is known about their finite-time analysis (or the rate of convergence). Indeed, it is not obvious how to derive such convergence rate by using {\sf ODE} approach. A concentration bound was given in \cite{Thoppe2018_ConcentrationBound,borkar2008} for the {\sf SA} algorithm under a strict stability assumption of the iterates. Recently, a finite-time analysis of {\sf TD}$(0)$ algorithm with linear function approximation was simultaneously studied in \cite{Dalal2018_FiniteTD0,Bhandari2018_FiniteTD0} for a single agent problem. These works carefully characterize the progress of {\sf TD}$(0)$ update and derive its convergence rate by utilizing the standard techniques of {\sf SGD} and the results in \cite{Tsitsiklis1997_TD}. 

Within the context of {\sf MARL}, an asymptotic convergence of the distributed gossiping {\sf TD}$(0)$ with linear function approximation was probably first studied in \cite{Mathkar2017_GossipRL}, where the authors utilize the standard techniques of {\sf ODE} approach. Such results were also studied implicitly in the context of distributed actor-critic methods in \cite{zhang2018_ICML}. On the other hand, unlike recent works about finite-time analysis in a single agent setup \cite{Dalal2018_FiniteTD0,Bhandari2018_FiniteTD0}, the rate of convergence of distributed {\sf TD}$(0)$ is missing in the existing literature of {\sf MARL}, which is the focus of this paper.     

Finally, we mention some related {\sf RL} methods for solving policy evaluation problems in both single agent {\sf RL} and {\sf MARL}, such as, the gradient temporal difference methods studied in \cite{Sutton2008_GTD,Sutton2009_FGTD,Liu2015_FiniteGTD,Macua_2015_TAC,Stankovic2016_ACC,Wai2018_NIPS}, least squares temporal difference ({\sf LSTD}) \cite{Bradtke1996_LSTD,Tu2018_ICML}, and least squares policy evaluation ({\sf LSPE}) \cite{Nedic2003_LSPE,Yu2009_TAC}. Although they share some similarity with {\sf TD} learning, these methods belong to a different class of algorithms, which involve more iteration complexity in their updates. 

\subsection{Main Contributions}
In this paper, we study a distributed variant of the {\sf TD}(0) algorithm for solving a policy evaluation problem in {\sf MARL}. Our distributed algorithm is composed of the popular consensus step and local {\sf TD}$(0)$ updates at the agents. Our main contribution is to provide a finite-time analysis for the convergence of distributed {\sf TD}$(0)$ over time-varying networks. We obtain an explicit upper bound on the rate of convergence of this algorithm as a function of the network topology and the discount factor. Our results mirror what we would expect from using distributed {\sf SGD} for solving convex optimization problems. For example, when the stepsizes are chosen independently with the problem's parameters, the function value estimated at each agent's time-weighted estimates converges to a neighborhood around the optimal value at a rate $\mathcal{O}(1\,/\,k)$ under constant stepsizes and asymptotically converges to the optimal value at a rate $\mathcal{O}(1\, /\,\sqrt{k+1})$ under time-varying stepsizes. Moreover, our rates also show the dependence on the network topology and the discount factor associated with the accumulative reward. These convergence rates mirrors the ones from using distributed stochastic gradient descent for solving convex optimization problems. On the other hand, we observe the same results in the case of strongly convex optimization problems for both constant and time-varying stepsizes, when the stepsizes are chosen based on some knowledge of the problem's parameter. We note that such an explicit formula for the rate of distributed {\sf TD}$(0)$ algorithm is not available in the literature.


\section{Centralized Temporal-Difference Learning}\label{sec:centralized}
We briefly review here the problem of policy evaluation for a given stationary policy $\mu$ over a Markov Decision Process ({\sf MDP}). This will facilitate our development of multi-agent reinforcement learning in the next section. We consider a discounted reward {\sf MDP} defined by $5$-tuple $(\Scal,\Ucal,\Pcal,\Rcal,\gamma)$, where $\Scal$ is a finite set of states,   $\Scal = \{1,\ldots,n\}$. In addition, $\Ucal$ is the set of control actions, $\Pcal$ is the set of transition probability matrices  associated with the Markov chain, $\Rcal$ is the reward function, and $\gamma\in(0,1)$ is the discount factor. 

At each time $k\geq0$, the agent observes the current state $s(k) = i$ and applies an action $\mu(s(k))$, where $\mu:\Scal\rightarrow\Ucal$. The system then moves to the next state 
$s'(k) = j$ with some probability $p_{ij}(\mu(i))$ decided by the action $\mu(i)$. Moreover, the agent receives the instantaneous reward $r(k)$. That is, for each transition from $i$ to $j$ an immediate reward $r$ is observed according to $\Rcal(i,j)$. In the sequel, since the policy is stationary, we drop $\mu$ in our notation for convenience. The discounted accumulative reward $J^*:\Scal\rightarrow\Rset$ associated with this Markov chain is defined for all $i\in\Scal$ as\vspace{-0.2cm}
\begin{align}
J^*(i) \triangleq \Eset \left[\sum_{k=0}^{\infty}\gamma^{k}\Rcal(s(k),s'(k))\Bigm| s(0) = i\right],\label{centralized:ValueFunc}
\end{align}
which is the solution of the following Bellman equation \cite{Sutton1998_book,Bertsekas1999_book}\vspace{-0.2cm}
\begin{align}
J^*(i) = \sum_{j = 1}^n p_{ij}\left[\Rcal(i,j) + \gamma J^*(j) \right],\quad \forall i\in\Scal.\label{centralized:Bellman}
\end{align}
We are interested in the case when the number of states is very large, and so computing $J^*$ exactly may be intractable.  To mitigate this, we use low-dimensional approximation $\tilde{J}$ of $J^*$, restricting $\tilde{J}$ to be in a linear subspace.  While more advanced nonlinear approximations using, for example, neural nets as in the recent works \cite{Mnih2015_Nature, Silver2016_MasteringTG} may lead to more powerful approximations, the simplicity of the linear model allows us to analyze it in detail.  The linear function approximation $\tilde{J}$ is parameterized by a weight vector $\theta\in\Rset^{K}$, with\vspace{-0.2cm}
\begin{align}
\tilde{J}(i,\theta) = \sum_{\ell=1}^{K}\theta_\ell\phi_{\ell}(i),\label{centralized:LinearAprox}
\end{align}  
for a given set of $K$ feature vectors $\phi_{\ell}:\Scal\rightarrow\Rset$, $\ell\in\{1,\ldots,K\}$, where $K\ll n$. Let $\phi(i)$ be a vector defined as
\begin{align*}
\phi(i) = (\phi_1(i),\ldots,\phi_K(i))^T\in\Rset^{K}.
\end{align*}
And let $\Phi\in\Rset^{n\times K}$ be a matrix, whose $i-$th row is the row vector $\phi(i)^T$ and whose $\ell-$th column is the vector $\phi_\ell = (\phi_\ell(1),\ldots,\phi_\ell(n))^T\in\Rset^{n}$, that is 
\begin{align*}
\Phi = \left[\begin{array}{ccc}
\vert     &  & \vert\\
\phi_1     & \cdots & \phi_K\\
\vert & & \vert
\end{array}\right] = \left[\begin{array}{ccc}
\mbox{---}   & \phi(1)^T & \mbox{---}\\
\cdots     & \cdots & \cdots\\
\mbox{---}   & \phi(n)^T & \mbox{---}
\end{array}\right].
\end{align*}
Thus, $\tilde{J}(\theta) =  \Phi \theta$, giving the gradient of $\tilde{J}$ w.r.t $\theta$ as
\begin{align*}
\nabla\tilde{J} = \Phi^T\quad\text{and}\quad \nabla\tilde{J}(i,\theta) = \phi(i),\quad \forall i\in\Scal. 
\end{align*}
The goal now is to find a $\tilde{J}$ that is the best approximation of $J^*$  based on the generated data by applying the stationary policy $\mu$ on the {\sf MDP}. That is, we seek an optimal weight $\theta^*$ such that the distance between $\tilde{J}$ and $J^*$ is minimized. For solving this problem, we are interested in the {\sf TD}$(0)$ algorithm, which is equivalent to a stochastic approximation for solving the Bellman equation \eqref{centralized:Bellman} \cite{Bertsekas1999_book}. In particular, we assume that at each time $k$ there is an oracle giving one data tuple $(s(k),s'(k),r(k))$, probably through simulation. The method of TD$(0)$ then updates $\theta$ as\vspace{-0.1cm} 
\begin{align}
\theta(k+1) = \theta(k) + \alpha(k) d(k)\nabla\tilde{J}(s(k),\theta(k)),\label{centralized:TD0} 
\end{align}
where $d(k)\in\Rset$ is the temporal difference at time $k$\vspace{-0.1cm}
\begin{align*}
d(k) &= r(k) + \gamma\tilde{J}(s'(k),\theta(k))- \tilde{J}(s(k),\theta(k)).    
\end{align*}
Here, $d(k)$ represents the difference between the outcome $r(k) + \gamma\tilde{J}(s'(k),\theta(k))$ of the current stage and the current estimate $\tilde{J}(s(k),\theta(k))$. Thus, $d(k)$ provides us an indicator whether to increase or decrease our current variable $\theta(k)$. 

In \cite{Tsitsiklis1997_TD}, to study the convergence of {\sf TD}$(0)$, the authors viewed $J^*$ as a fixed point of the Bellman operator $\Tcal:\Rset^n\rightarrow\Rset^n$ defined as\vspace{-0.1cm}
\begin{align*}
(\Tcal J)(i) =  \sum_{j=1}^n p_{ij} \left\{\Rcal(i,j)+\gamma J(j) \right\},\quad \forall i\in\Scal.
\end{align*}
and showed that $\{\theta(k)\}$ generated by {\sf TD}$(0)$ converges to $\theta^*$ almost surely, where $\theta^*$ is the unique solution of the projected Bellman equation $\Pi\, \Tcal(\Phi\theta^*) = \Phi\theta^*$, and satisfies\vspace{-0.1cm} 
\begin{align}
\|\Phi \theta^* - J^*\|_{D} \leq \frac{1}{1-\gamma}\|\Pi\, J^* - J^*\|_D.\label{MARL:opt_bound}   
\end{align}
Here, $\Pi\,J$ denotes the projection of a vector $J$ to the linear subspace spanned by the feature vectors $\phi_\ell$\vspace{-0.1cm}
\begin{align*}
\Pi\,J = \underset{y\in\text{span}\{\phi_\ell\}}{\arg\min}\|y - J\|_D,    
\end{align*}
where $\|J\|_{D}^2 = J^TDJ$ is the weighted norm of $J$ associated with the $n\times n$ diagonal matrix $D$, whose diagonal entry are $(\pi(1),\ldots,\pi(n))$, the stationary distribution associated with $\Pcal$. Moreover, we denote by $\|\cdot\|$ the Euclidean and Frobenius norms for a vector and a matrix, respectively.  

As shown in \cite{Tsitsiklis1997_TD} $\theta^*$ satisfies $\Abf \theta^* =  b$, where $\Abf$ is positive definite, i.e., $x^T\Abf x > 0 $ $\forall x$, and
\begin{align}
\hspace{-0.3cm}\Abf = \Eset_{\pi}\left[\phi(s)(\phi(s)-\gamma\phi(s'))^T\right],\quad b = \Eset_{\pi}[r\phi(s)].\label{notation:A_b} 
\end{align}
It is worth noting that although {\sf TD}(0) can be viewed as a stochastic approximation method for solving \eqref{centralized:Bellman}, it is not a {\sf SGD} method since the temporal direction $d(k)\nabla\tilde{J}(s(k),\theta(k))$ is not a true stochastic gradient of any static objective function. This makes analyzing the finite-time convergence of {\sf TD}$(0)$ more challenging since standard techniques of {\sf SGD} are not applicable. Our focus, therefore, is to provide such a finite-time analysis for the distributed variant of {\sf TD}(0) algorithm in the context of {\sf MARL}.



\section{Multi-Agent Reinforcement Learning}
We consider a multi-agent reinforcement learning system of $N$ agents modeled by a Markov decision process. We assume that the agents can communicate with each other through a given sequence of time-varying undirected graphs $\Gcal(k) = (\Vcal,\Ecal(k))$, where $\Vcal = \{1,\ldots,N\}$ and $\Ecal(k) = \Vcal\times\Vcal$ are the vertex and edge sets at time $k$, respectively.  This framework can be mathematically characterized by a $6$-tuple $(\Scal,\{\Ucal_v\},\Pcal, \{\Rcal_v\},\gamma,\Gcal(k))$ for $v\in\Vcal$ at each time step $k$. Here, $\Scal=\{1,\ldots,n\}$ is the global finite state space observed by the agents, $\Ucal_v$ is the set of control available to each agent $v$, and $\Rcal_v$ is each agent's reward function. 

At any time $k$, each agent $v$ observes the current states $s(k)$ and applies an action $\mu_v(k)\in\Ucal_v$, where $\mu_v$ is a stationary policy of agent $v$. Based on the joint actions of the agents, the system moves to the new state $s'(k)$ and agent $v$ receives an instantaneous local reward $r_v(k)$, defined by $\Rcal_v(s(k),s'(k))$ for each  transition of states. The goal of the agents is to cooperatively find the total accumulative reward $J^{*}$ over the network defined as\vspace{-0.1cm} 
\begin{align}
\hspace{-0.2cm}J^{*}(i) \triangleq \Eset \left[\sum_{k=0}^{\infty}\frac{\gamma^{k}}{N}\sum_{v\in\Vcal}\Rcal_v(s(k),s'(k))\Bigm| s(0) = i\right],  \label{distributed:ValueFunc}
\end{align}
which also satisfies the following Bellman equation \vspace{-0.1cm}
\begin{align*}
J^{*}(i) = \sum_{j=1}^n p_{ij} \left\{\frac{1}{N}\sum_{v\in\Vcal}\Rcal_v(i,j)+\gamma J^{*}(j) \right\},\qquad i\in\Scal.
\end{align*} 
Similar to the centralized problem, we are interested in finding a linear approximation $\tilde{J}$ of $J^{*}$ as given in Eq.\ \eqref{centralized:LinearAprox}. In addition, since each agent knows only its own reward function, the agents have to cooperate to find $\tilde{J}$. In the following, for solving such problem we provide a distributed variant of the {\sf TD}(0) algorithm presented in Section \ref{sec:centralized}, where the agents only share their estimates of the optimal $\theta^*$ to its neighbors but not their local rewards. Similar to the centralized problem (a.k.a Eq.\ \eqref{notation:A_b}),  $\theta^*$ satisfies  \vspace{-0.1cm}
\begin{align}
\Abf\theta^* = \frac{1}{N} \sum_{v\in\Vcal} b_v,\label{sec_MARL_rates:opt_cond}\vspace{-0.4cm}
\end{align}
where the positive defnite matrix $\Abf$ and $b_v$ are defined as
\begin{align}
\hspace{-0.3cm}\Abf = \Eset_{\pi}\left[\phi(s)(\phi(s)-\gamma\phi(s'))^T\right],\quad b_v = \Eset_{\pi}[r_v\phi(s)].\label{notation:A_bv} 
\end{align}

\subsection{Distributed Consensus-Based TD$(0)$ Learning}\label{subsec:DistributedTD0}
In this section, we study a distributed consensus-based variant of the centralized {\sf TD}(0) method, formally stated in Algorithm \ref{alg:DTD0}. In particular, agent $v$ maintains their own estimate $\theta_v\in\Rset^{K}$ of the optimal $\theta^*$. At any iteration $k\geq0$, each agent $v$ only receives the estimates $\theta_u$ from its neighbors $u\in\Ncal_v(k)$, where $\Ncal_v(k):= \{ u \in \Vcal\; |\; (v,u) \in\Ecal(k)\}$ is the set of node $v$'s neighbors at time $k$. Agent $v$ then observes one data tuple $(s(k),s'(k),r_v(k))$ returned by the oracle, following an  update to its estimate $\theta_v(k+1)$ by using  Eq.\ \eqref{alg:theta_v}. Here, $W_{vu}(k)$ is the weight which agent $v$ assigns for $\theta_u(k)$. Finally, $[\cdot]_{\Xcal}$ denotes the projection to a set $\Xcal\subset\Rset^{K}$, whose condition is given shortly.


The update of Eq.\ \eqref{alg:theta_v} has a simple interpretation: agent $v$ first computes $y_v$ by forming a weighted average of its own value $\theta_v$ and the values $\theta_u$ received from its neighbor $u\in\Ncal_v$, with the goal of seeking consensus on their estimates. Agent $v$ then moves along its own temporal direction $d_v(k)\nabla\tilde{J}(s(k),\theta_v(k))$ to update its estimate, pushing the consensus point toward $\theta^*$. In Eq.\ \eqref{alg:theta_v} each agent $v$ only shares $\theta_v$ with its neighbors but not its immediate reward $r_v$. In a sense, the agents implement in parallel $N$ local {\sf TD}(0) methods and then combine their estimate through consensus steps to find the global approximate reward $\tilde{J}$. 

\begin{algorithm}[t]
\caption{Distributed {\sf TD}(0) Algorithm}
\begin{enumerate}[leftmargin = 4mm]
\item \textbf{Initialize}: Each agent $v$ arbitrarily initializes $\theta_v(0)\in\Xcal$ and the sequence of stepsizes $\{\alpha(k)\}_{k\in\Nset}$.\\
Set $\hat{\theta}_v(0) = \theta_v(0)$ and $S_v(0) = 0$.\vspace{-0.2cm}
\item \textbf{Iteration}: For $k=0,1,\ldots,$ agent $v\in\Vcal$ implements\vspace{-0.2cm}
\begin{itemize}[leftmargin = 4mm]
    \item[a.] Exchange $\theta_v(k)$ with agent $u\in\Ncal_{v}(k)$
    \item[b.] Observe a tuple $(s(k),s'(k),r_v(k))$
    \item[c.] Execute local updates \vspace{-0.2cm}
    {\small
        \begin{align}
        \begin{aligned}
        &y_v(k) = \sum_{u\in\Ncal_v(k)}W_{vu}(k)\theta_u(k)\\ 
        &d_v(k) = r_v(k) + \theta_v(k)^T\Big(\gamma\phi(s'(k)) - \phi(s(k))\Big)\\
        &\theta_v(k+1) =\Big[y_v(k) + \alpha(k) d_v(k)\phi(s(k))\Big]_{\Xcal}
        \end{aligned}
        \label{alg:theta_v}
        \end{align}\vspace{-0.2cm}
    }\vspace{-0.2cm}
    \item[d.] Update the output\vspace{-0.2cm}
    \begin{align*}
     S_v(k+1) &= S_v(k) + \alpha(k)\notag\\
    \hat{\theta}_v(k+1) &= \frac{S(k)\hat{\theta}_v(k)+\alpha(k)\theta_v(k)}{S(k+1)}
    \end{align*}
\end{itemize}\vspace{-0.3cm}
\end{enumerate}
\label{alg:DTD0}
\end{algorithm}

\subsection{Convergence Rates of Distributed {\sf TD}$(0)$}
We state here the main results of this paper, the convergence rates of the distributed {\sf TD}$(0)$ algorithm.  In particular, we provide an explicit formula for the upper bound on the rates of {\sf TD}$(0)$ for both constant and diminishing stepsizes. Our bounds mirror the results that we would expect from the ones using distributed {\sf SGD} for solving convex optimization problems. For ease of exposition, we delay the analysis of these results to Sections \ref{sec_analysis:proof_thm1} and \ref{sec_analysis:proof_thm2}. 

Our main results are established based on the assumption that the data tuple $\{(s(k),s'(k),r_v(k)\}$ are sampled i.i.d from stationary distribution for all $k$ and $v$. However, within a tuple, $s'(k)$ and $r_v(k)$ are dependent on $s(k)$. We note that the i.i.d condition is often assumed in the literature when dealing with the rates of {\sf RL} algorithms, see for example; \cite{ Dalal2018_FiniteTD0,Bhandari2018_FiniteTD0}. Such a condition is not easy to remove since the dependence between samples can make the analysis become extremely complicated in general. One possible way to collect i.i.d samples is to generate independently a number of trajectories and using first-visit methods, see \cite{Bertsekas1999_book}. On the other hand, sampling from stationary distribution can be done by taking last samples of a long trajectory.     

Moreover, we make the following fairly standard assumptions in the existing literature of consensus and reinforcement learning  \cite{Tsitsiklis1997_TD,Dalal2018_FiniteTD0,Bhandari2018_FiniteTD0,NedicOR2018}. To the rest of this paper, we will assume that these assumptions always hold. 
\begin{assump}\label{assump:BConnectivity}
There exists an integer $\Bcal$ such that the following graph is connected for all positive integers $\ell$\vspace{-0.25cm}
\begin{align*}
(\Vcal\,,\,\Ecal(\ell\Bcal)\cup\Ecal(\ell\Bcal+1)\ldots\cup\Ecal((\ell+1)\Bcal-1)).
\end{align*}
\end{assump} 
\begin{assump}\label{assump:doub_stoch} There exists a positive constant $\beta$ such that $\Wbf(k)=[W_{vu}(k)]\in\Rset^{N\times N}$ is doubly stochastic and $W_{vv}(k)\geq \beta$ $\forall v\in\Vcal$. Moreover, $W_{vu}(k)\in[\beta,1)$ if $(v,u)\in\Ncal_v(k)$ otherwise $W_{vu}(k) = 0$ for all $v,u\in\Vcal$.
\end{assump}
\begin{assump}\label{assump:Markov}
The Markov chain associated with $\Pcal$ is irreducible. 
\end{assump}
\begin{assump}\label{assump:reward}
All the local rewards are uniformly bounded, i.e., there exist constants $C_v$, for all $v\in\Vcal$ such that $|\,\Rcal_v(s,s')\,|\leq C_v$, for all $s,s'\in\Scal$. 
\end{assump}
\begin{assump}\label{assump:features}
The feature vectors $\{\phi_{\ell}\}$, for all $\ell\in\{1,\ldots,K\}$, are linearly independent, i.e., the matrix $\Phi$ has full column rank. In addition, we assume that all feature vectors $\phi(s)$ are uniformly bounded, i.e., $\|\phi(s)\|\leq 1$.  
\end{assump}
\begin{assump}\label{assump:Xset}
The convex compact set $\Xcal\subset\Rset^{K}$ contains the fixed point $\theta^*$ of the projected Bellman equation.
\end{assump}\vspace{-0.2cm}
Assumption \ref{assump:BConnectivity}  ensures the long-term connectivity and information propagation between the agents, while Assumption \ref{assump:doub_stoch} imposes the underlying topology of $\Gcal(k)$ where each agent only communicates with its neighbors. Assumptions \ref{assump:BConnectivity} and \ref{assump:doub_stoch} yield the following condition \cite{NedicOR2018}\vspace{-0.1cm}
\begin{align}
\left\|\Wbf(k)\ldots\Wbf(k+\Bcal-1)\Qbf\Theta\right\|\leq \eta\|\Qbf\Theta\|.\;\forall \Theta,\label{const:Bmix}
\end{align}
In Eq.\ \eqref{const:Bmix}, $\Theta\in\Rset^{N\times K}$ and $\Qbf\in\Rset^{N\times N}$ are defined as
\begin{align}
\Theta \triangleq \left[\begin{array}{ccc}
\mbox{---}   & \theta_1^T & \mbox{---}\\
\cdots     & \cdots & \cdots\\
\mbox{---}   & \theta_N^T & \mbox{---}
\end{array}\right],\quad \Qbf = \Ibf - \frac{1}{N} \1\1^T,\label{notation:Theta_Q}
\end{align}
where $\Ibf$ and $\1$ are the identity matrix and the vector in $\Rset^{N}$ with all entries equal to $1$, respectively. Moreover, denote by $\sigma_2(\Wbf(k))$ the second largest singular value of $\Wbf(k)$ and $\eta\in(0,1)$ a parameter representing the spectral properties of the sequence of graphs $\{\Gcal(k)\}$  defined as\vspace{-0.2cm}
\begin{align}
\eta = \min\;\left\{1-1/(2N^3),\, \sup_{k\geq 0}\sigma_2(\Wbf(k))\right\}.\label{const:delta}
\end{align}
For convinience, we define $\delta := \eta^{\frac{1}{\Bcal}}$.
Assumption \ref{assump:Markov} guarantees that there exists a unique stationary distribution $\pi$ with positive entries, while under Assumption \ref{assump:reward} the accumulative reward $J^*$ is well defined. Under Assumption \ref{assump:features}, the projection operator $\Pi$ is well defined. If there are some dependent $\phi_{\ell}$, we can simply disregard those dependent feature vectors. Moreover, the uniform boundedness of $\phi_\ell$ can be guaranteed through feature normalization. 

Finally, Assumption \ref{assump:Xset} is used to guarantee the stability of agents' updates, which is often assumed in the literature of {\sf MARL} and stochastic approximation, see for example; \cite{zhang2018_ICML,borkar2008}. We note that this projection step is only used for the purpose of our convergence analysis. In practice, we may not need this step to implement Algorithm \ref{alg:DTD0} since the consensus step likely keeps the agents' estimates close to each other while the {\sf TD} direction drives these estimates to an optimal solution.\vspace{-0.2cm}          

Denote by $\sigma_{\min}$ and $\sigma_{\max}$ the smallest and largest singular value of $\Abf$, respectively. Let ${
R_0 = \max_{\theta\in\Xcal}\|\theta-\theta^*\|.}$ We now present our first results,   the convergence rates of the approximate value function estimated at each agent's output to the optimal value. That is, we provide the speed of convergence of $\tilde{J}(\hat{\theta}_v(k))$ to $\Phi\theta^*$, for each $v\in\Vcal$. These results are established based on proper conditions on stepsizes $\alpha(k)$ chosen independently of the problem's parameters.   
\begin{thm}\label{thm:opt_value}Let $\theta_v(k)$, for all $v\in\Vcal$, be generated by Algorithm \ref{alg:DTD0}. In addition, given the constant $L>0$ in Lemma \ref{lem:consensus}, let $\beta_0$ and $\beta_1$ be two positive constants defined as
\begin{align}
&\beta_0 = \Eset\left[\|\btheta(0)-\theta^*\|^2\right]\notag\\
&\qquad\quad +\frac{\alpha(0)\Eset\left[\|\Theta(0)\|\right]\left(L+2N\sigma_{\max}R_0\right)}{N\eta(1-\delta)}\notag\\ 
&\beta_1 = \frac{4L(L+N\sigma_{\max}R_0)}{N\eta(1-\delta)}\cdot\label{thm_opt_value:beta}
\end{align}
1. If $\alpha(k) = \alpha$ for some positive constant $\alpha$ then $\forall v\in\Vcal$
\begin{align}
\hspace{-0.3cm}\|\tilde{J}(\hat{\theta}_v(k))-\tilde{J}(\theta^*)\|_D^2\leq \frac{\beta_0}{\alpha(1-\gamma)}\frac{1}{k+1} + \frac{\beta_1\alpha}{(1-\gamma)}\cdot \label{thm_opt_value:Ineq1}
\end{align}
2. If $\{\alpha(k)\} = 1\,/\,\sqrt{k+1}$ for all $k\geq0$ then $\forall v\in\Vcal$
\begin{align}
&\|\tilde{J}(\hat{\theta}_v(k))-\tilde{J}(\theta^*)\|_D^2\leq \frac{\beta_0+\beta_1(1+\ln(k+1))}{2(1-\gamma)\sqrt{k+1}}\cdot\label{thm_opt_value:Ineq2}
\end{align}
\end{thm}
As shown in Eq. \eqref{thm_opt_value:Ineq1}, our rate mirrors what we would expect in using distributed {\sf SGD} for solving a convex optimization problem with a constant stepsize, i.e., the convergence of the function value to a neighborhood around the optimal value occurs at $\mathcal{O}(1\,/\,k+1)$. In addition, the rate of the distributed {\sf TD}$(0)$ also depends inversely on $1-\gamma$ and $1-\delta$. Here, the term $1\,/\,(1-\gamma)$ is expected, as can been seen from Eq. \eqref{MARL:opt_bound}. Moreover, $1-\delta$ is the spectral gap of $\Wbf$ and its inverse represents the connectivity of the underlying communication graph between agents. For different graphs, we have different values of $\delta$, see for example \cite{NedicOR2018}. Similar observation holds for the case of time-varying stepsizes $\alpha(k) = 1\,/\,\sqrt{k+1}$, where we would expect an asymptotic rate at $\mathcal{O}(1\,/\,\sqrt{k+1})$, with the same dependence on the inverse of $1-\gamma$ and $1-\delta$.      

Second, we derive the convergence rate of $\hat{\theta}_v(k)$, for all $v\in\Vcal$, to the optimal solution $\theta^*$, where $\sigma_{\min}$ of $\Abf$ is assumed to be known. In particular, the stepsizes $\alpha(k)$ are chosen based on this $\sigma_{\min}$. We again observe the same rates as we would expect from using distributed {\sf SGD} for solving strongly convex optimization problems. In addition, these rates depend on the condition number $\sigma_{\max}/\sigma_{\min}$ of $\Abf$, as often observed in distributed {\sf SGD}. 
\begin{thm}\label{thm:opt_theta}
Let $\theta_v(k)$, for all $v\in\Vcal$, be generated by Algorithm \ref{alg:DTD0}. In addition, given the constant $L>0$ in Lemma \ref{lem:consensus}, let $\beta_2$ and $\beta_3$ be two positive constants defined as\vspace{-0.2cm}
\begin{align}
\begin{aligned}
&\beta_2 = \frac{4(L+N\sigma_{\max}R_0)\Eset\left[\|\Theta(0)\|\right]}{N\eta}\\ 
&\beta_3 = \frac{16L(L+N\sigma_{\max}R_0)}{N\eta(1-\delta)}\cdot\label{thm_opt_theta:beta1}
\end{aligned}
\end{align}\vspace{-0.2cm}
$1$ If $\alpha(k) = \alpha\in(0,1/\sigma_{\min})$ then $\forall v\in\Vcal$
\begin{align}
\hspace{-0.4cm}\Eset[\|\theta_v(k)-\theta^*\|^2]&\leq 2\Eset\Big(\|\btheta(0)-\theta^*\|^2+2\|\Theta(0)\|^2\Big)\rho^{k}\notag\\ 
& + \frac{\beta_2}{1-\rho}\alpha + \frac{\beta_3}{(1-\rho)(1-\delta)}\alpha^2,
\label{thm_opt_theta:Ineq1}
\end{align}
where $\rho = \max\{1-\sigma_{\min}\alpha,\;\delta\}\in(0,1)$.\\
$2$ If $\alpha(k) = \alpha_0\,/\,(k+1)$ where $\alpha_0 > 1\,/\,\sigma_{\min}$ then $\forall v\in\Vcal$
\begin{align}
&\Eset\left[\|\hat{\theta}_v(k)-\theta^*\|^2\right]\notag\\
&\leq \left(\frac{\beta_2}{2\sigma_{\min}(1-\delta)}+\frac{\alpha_0\beta_3}{4\sigma_{\min}}\right)\frac{\ln(k+1)}{k+1}\cdot\label{thm_opt_theta:Ineq2}
\end{align}
\end{thm}

\section{Finite-Time Analysis of Distributed {\sf TD}(0)}
In this section, our goal is to provide the proofs of the main results in this paper, that is, the proofs of Theorems \ref{thm:opt_value} and \ref{thm:opt_theta}. We start by introducing more notation and stating some preliminary results corresponding to the updates of consensus and {\sf TD} steps.  

\subsection{Notation and Preliminary Results}\label{subsec:notation_assump}
Using $\Abf,b_v$ are given in Eq.\ \eqref{notation:A_bv} we denote by 
\begin{align}
\begin{aligned}
h_v(k) &= b_v -  \Abf\theta_v(k)\\
M_v(k) &= d_v(k)\phi(s(k)) - [b_v - \Abf\theta_v(k)],
\end{aligned}\label{analysis:hM}
\end{align}
Then, we rewrite Eq.\ \eqref{alg:theta_v} as
\begin{align}
\begin{aligned}
y_v(k) &= \sum_{u\in\Ncal_v(k)}W_{vu}(k)\theta_u(k)\\
\tilde{\theta}_v(k) &= y_v(k) + \alpha(k)(h_v(k)+M_v(k))\\
e_v(k)&= \tilde{\theta}_v(k) - [\tilde{\theta}_v(k)]_{\Xcal}\\
\theta_v(k+1)  &=\left[\tilde{\theta}_v(k)\right]_{\Xcal} = \tilde{\theta}_v(k) - e_v(k),
\end{aligned}
\label{analysis:theta_k}
\end{align}
Thus, using $\Wbf(k)$ in Assumption \ref{assump:doub_stoch} and $\Theta$ in Eq.\ \eqref{notation:Theta_Q}, the matrix form of Eq.\ \eqref{analysis:theta_k} is
\begin{align}
\begin{aligned}
&\Ybf(k) = \Wbf(k)\Theta(k)\\
&\tilde{\Theta}(k) = \Wbf(k)\Theta(k) + \alpha(k)(\Hbf(k) + \Mbf(k))\\ 
&\Ebf(k) = \tilde{\Theta}(k) - [\tilde{\Theta}(k)]_{\Xcal}\\
&\Theta(k+1)  = \tilde{\Theta}(k) - \Ebf(k),
\end{aligned}
\label{analysis:Theta_k} 
\end{align}
where $\Hbf(k)$, $\Mbf(k)$, and $\Ebf(k)$
are the matrices, whose $v-$th rows are $h_v(k)^T$, $M_v(k)^T$, and $e_v(k)^T$, respectively. Moreover, $[\tilde{\Theta}(k)]_{\Xcal}$ is the row-wise projection of $\tilde{\Theta}(k)$. Given the vectors $\theta_v$ we denote by $\btheta$ their average, i.e., 
$\btheta \triangleq 1\,/N\,\sum_{v\in\Vcal}\theta_v.$ Thus, Assumption \ref{assump:doub_stoch} and Eq.\ \eqref{analysis:Theta_k} gives\vspace{-0.1cm}
\begin{align}
&\btheta(k+1) = \btheta(k) +\alpha(k)(\bar{h}(k)+\bar{m}(k)) - \bar{e}(k),\label{analysis:theta_bar}
\end{align} 
Finally, we provide here some preliminary results, which are useful to derive our main results in the next section.  For convenience, we put their proofs in the supplementary document. We first provides an upper bound for the consensus error defined at time $k$ as $ \Theta(k) - \1\btheta(k)^T = \Qbf\Theta(k)$ in the following lemma, where $\Qbf$ is given in Eq.\ \eqref{notation:Theta_Q}.
\begin{lem}\label{lem:consensus} Let $\theta_v(k)$, for all $v\in\Vcal$, be generated by Algorithm \ref{alg:DTD0}. Let $\{\alpha(k)\}$ be a nonnegative nonincreasing sequence of stepsizes. Then there exists a constant $L>0$ such that\\ 
$1.$ The consensus error $\Qbf\Theta(k)$ satisfies\vspace{-0.2cm}
\begin{align}
\|\Qbf\Theta(k)\| \leq \delta^{k}\frac{\|\Theta(0)\|}{\eta} + \frac{2L}{\eta}\sum_{t=0}^{k-1}\delta^{k-1-t}\alpha(t).\label{lem_consensus:MainIneq}\vspace{-0.2cm}
\end{align}
$2.$ In addition, we obtain \vspace{-0.2cm}
\begin{align}
\hspace{-0.35cm}\sum_{t=0}^{k}\alpha(t)\|\Qbf\Theta(t)\| \leq \frac{\alpha(0)\|\Theta(0)\|}{\eta(1-\delta)} + \frac{2L}{\eta(1-\delta)}\sum_{t=0}^k\alpha^2(t).    \label{lem_consensus:FiniteBound}
\end{align}
\end{lem}\vspace{-0.1cm}
Second, we provide an upper bound in expectation for the optimal distance 
$\|\btheta(k) - \theta^*\|$.
\begin{lem}\label{lem:opt}
Let $\theta_v(k)$, for all $v\in\Vcal$, be generated by Algorithm \ref{alg:DTD0}. In addition, let $\{\alpha(k)\}$ be a nonnegative nonincreasing sequence of stepsizes. Then we have
\begin{align}
&\Eset\left[\|\btheta(k+1) - \theta^*\|^2\right]\notag\\ 
&\leq \Eset[\|\btheta(k) - \theta^*\|^2]+2\alpha(k)\Eset[(\btheta(k) - \theta^*)^T\bar{h}(k)]\notag\\
&\qquad + \frac{4L^2\alpha^2(k)}{N}+  \frac{2L}{N}\alpha(k)\Eset[\left\|\Qbf\Theta(k)\right\|].\label{lem_opt:Ineq}    
\end{align}
\end{lem}

\subsection{Proof of Theorem \ref{thm:opt_value} }\label{sec_analysis:proof_thm1} 

By Eq.\ \eqref{sec_MARL_rates:opt_cond} we have  $\bar{b} = \Abf\theta^*$. Thus, Eq.\ \eqref{analysis:hM} gives 
\begin{align}
&\bar{h}(k) = \bar{b} -  \Abf\btheta(k) = \Abf(\theta^* - \btheta(k))\notag\\
&= \Abf(\theta^* - \theta_u(k)) + \Abf(\theta_u(k)-\btheta(k)).\label{thm_opt_value:Eq1a}
\end{align}
Recall that  $\tilde{J}(\btheta(k)) = \Phi\btheta(k)$. In addition, since the data are sampled i.i.d from the stationary distribution, we have $\Phi^T D\Phi = \Eset\left[\phi(s)\phi(s)^T\right]$. Thus, given a $\theta\in\Rset^K$ consider
\begin{align}
&\|\tilde{J}(\theta)-\tilde{J}(\theta^*)\|_D^2 = (\theta-\theta^*)^T\Phi^T D\Phi(\theta-\theta^*)\notag\\
&= \Eset\left[(\theta-\theta^*)^T\phi(s)\phi(s)^T(\theta-\theta^*)\right]\notag\\ 
&= \Eset\left[\|\phi(s)^T(\theta-\theta^*)\|^2\right].\label{thm_opt_value:Eq1b}
\end{align}
Fix an index $u\in\Vcal$. Using Eq.\ \eqref{thm_opt_value:Eq1a} we consider 
\begin{align}
&\Eset\left[(\btheta(k)-\theta^*)^T\bar{h}(k)\right]\notag\\ 
&= \Eset\left[(\theta_u(k)-\theta^*)^T\bar{h}(k)\right]+\Eset\left[(\btheta(k)-\theta_u(k))^T\bar{h}(k)\right]\notag\\
&\stackrel{\eqref{thm_opt_value:Eq1a}}{=} \Eset\left[(\theta_u(k)-\theta^*)^T\Abf(\theta^*-\theta_u(k)) \right]\notag\\
&\qquad + 2 \Eset\left[(\theta_u(k)-\theta^*)^T\Abf(\theta_u(k)-\btheta(k)) \right]\notag\\
&\qquad + \Eset\left[(\btheta(k)-\theta_u(k))^T\Abf(\theta_u(k)-\btheta(k)) \right]\notag\\
&\leq \Eset\left[(\theta_u(k)-\theta^*)^T\Abf(\theta^*-\theta_u(k)) \right]\notag\\
&\qquad + 2 \Eset\left[(\theta_u(k)-\theta^*)^T\Abf(\theta_u(k)-\btheta(k)) \right]\notag\\
&\leq \Eset\left[(\theta_u(k)-\theta^*)^T\Abf(\theta^*-\theta_u(k)) \right]\notag\\
&\qquad + 2\sigma_{\max}R_0 \Eset\left[\|\theta_u(k)-\btheta(k)\| \right], \label{thm_opt_value:Eq2a}
\end{align}
where the first inequality is due to $\Abf$ is positive definite. Using the definition of $\Abf$ in Eq.\ \eqref{notation:A_bv} and Eqs.\ \eqref{thm_opt_value:Eq1a} and \eqref{thm_opt_value:Eq1b} we analize the first term on the right-hand side of Eq.\ \eqref{thm_opt_value:Eq2a}
{\small
\begin{align*}
&\Eset\left[(\theta_u(k)-\theta^*)^T\Abf(\theta^*-\theta_u(k)) \right]\notag\\
&= \Eset\left[(\theta_u(k)-\theta^*)^T\phi(s)\Big(\phi(s)-\gamma\phi(s')\Big)^T(\theta^*-\theta_u(k))\right]\notag\\
&= -\Eset\left[(\theta_u(k)-\theta^*)^T\phi(s)\phi(s)^T(\theta_u(k)-\theta^*)\right]\notag\\
&\quad + \gamma\Eset\left[(\theta_u(k)-\theta^*)^T\phi(s)\phi(s')^T(\theta_u(k)-\theta^*)\right]\notag\\
&\leq -\Eset\left[\|\phi(s)^T(\theta_u(k)-\theta^*)\|^2\right]\notag\\
&+ \gamma\sqrt{\Eset\left[\|\phi(s)^T(\theta_u(k)-\theta^*)\|^2\right]}\sqrt{\Eset\left[\|\phi(s')^T(\theta_u(k)-\theta^*)\|^2\right]}\notag\\
&= -(1-\gamma)\Eset\left[\|\phi(s)^T(\theta_u(k)-\theta^*)\|^2\right]\notag\\
&\stackrel{\eqref{thm_opt_value:Eq1b}}{=} -(1-\gamma)\|\tilde{J}(\theta_u(k))-\tilde{J}(\theta^*)\|_D^2,
\end{align*}}
where the inequality is due to the Cauchy-Schwarz inequality.
Substituting the preceding relation into Eq.\ \eqref{thm_opt_value:Eq2a} and using $\|\theta_u(k)-\btheta(k)\|\leq \|\Qbf\Theta(k)\|$ we obtain
\begin{align}
\Eset\left[(\btheta(k)-\theta^*)^T\bar{h}(k)\right]
&\leq -(1-\gamma)\|\tilde{J}(\theta_u(k))-\tilde{J}(\theta^*)\|_D^2\notag\\
&\quad + 2\sigma_{\max}R_0 \Eset\left[\|\Qbf\Theta(k)\| \right].\label{thm_opt_value:Eq2b}
\end{align}
Using Eq.\ \eqref{thm_opt_value:Eq2b} into Eq.\ \eqref{lem_opt:Ineq} gives\vspace{-0.2cm}
\begin{align*}
&\Eset\left[\|\btheta(k+1)-\theta^*\|^2\right]\notag\\ 
&\leq \Eset\left[\|\btheta(k)-\theta^*\|^2\right]+ \frac{4L^2}{N}\alpha^2(k) \notag\\
&\qquad+ 2\left(\frac{L}{N}+2\sigma_{\max}R_0\right)\Eset\left[\alpha(k)\left\|\Qbf\Theta(k)\right\|\right] \notag\\
&\qquad -2(1-\gamma)\alpha(k)\|\tilde{J}(\theta_u(k))-\tilde{J}(\theta^*)\|_D^2.\vspace{-0.2cm}
\end{align*}
Rearranging and summing up both sides of the preceding relation over $k$ from $0$ to $K$ for some constant $K>0$ gives\vspace{-0.2cm}
\begin{align}
&2(1-\gamma)\sum_{k=0}^K\alpha(k)\|\tilde{J}(\theta_u(k))-\tilde{J}(\theta^*)\|_D^2\notag\\
&\leq \Eset\left[\|\btheta(0)-\theta^*\|^2\right] + \frac{4L^2}{N}\sum_{k=0}^{K}\alpha^2(k)\notag\\
&\qquad + \frac{2L+4N\sigma_{\max}R_0}{N}\sum_{k=0}^K\Eset\left[\alpha(k)\|\Qbf\Theta(k)\|\right]\notag\\
&\leq \Eset\left[\|\btheta(0)-\theta^*\|^2\right] +\frac{2\alpha(0)\Eset\left[\|\Theta(0)\|\right]\left(L+2N\sigma_{\max}R_0\right)}{N\eta(1-\delta)}\notag\\ 
&\qquad +\frac{8L(L+N\sigma_{\max}R_0)}{N\eta(1-\delta)}\sum_{k=0}^K\alpha^2(k),\label{thm_opt_value:Eq2}
\end{align}
where the last inequality is due to Eq.\ \eqref{lem_consensus:FiniteBound}. We now consider two choices of $\alpha(k)$ with  $\beta_0$ and $\beta_1$ as defined in Eq.\ \eqref{thm_opt_value:beta}. 

$1.$ Let $\alpha(k) = \alpha > 0$. Dividing Eq.\ \eqref{thm_opt_value:Eq2} by $2\alpha(1-\gamma)(K+1)$ and using the Jensen's inequality yields Eq.\ \eqref{thm_opt_value:Ineq1}.

$2.$ Let $\alpha(k) =  1\,/\,\sqrt{k+1}$. Using the integral test yields\vspace{-0.2cm}
\begin{align*}
\sum_{t=0}^{K}\alpha(k) \geq 2\sqrt{K+1},\quad \sum_{t=0}^{K}\alpha^2(k) \leq (1+\ln(K+1)).   
\end{align*}
Thus, dividing Eq.\ \eqref{thm_opt_value:Eq2} by ${~2(1-\gamma)\sum_{k=0}^{K}\alpha(k)}$ and using the Jensen's inequality give Eq.\ \eqref{thm_opt_value:Ineq2}.

\subsection{Proof of Theorem \ref{thm:opt_theta} }\label{sec_analysis:proof_thm2} 
Fix a $u\in\Vcal$. Note that $2(x-y)^T\Abf(y-z) = \|x-z\|_{\Abf}^2-\|x-y\|_{\Abf}^2-\|z-y\|_{\Abf}^2$ $\forall x,y,z$. Thus, Eq. \eqref{thm_opt_value:Eq1a} gives 
\begin{align*}
&2\Eset\left[(\btheta(k)-\theta^*)^T\bar{h}(k)\right]\notag\\ 
&= 2\Eset\left[(\theta_u(k)-\theta^*)^T\Abf(\theta^*-\btheta(k))\right]\notag\\
&\qquad +2\Eset\left[(\btheta(k)-\theta_u(k))^T\Abf(\theta^*-\btheta(k))\right]\notag\\
&= -\Eset\left[\|\btheta(k)-\theta^*\|_{\Abf}^2\right]-\Eset\left[\|\theta_u(k)-\theta^*\|_{\Abf}^2\right]\notag\\
&\qquad + \Eset\left[\|\btheta(k)-\theta_u(k)\|_{\Abf}^2\right]\notag\\ &\qquad+2\Eset\left[(\btheta(k)-\theta_u(k))^T\Abf(\theta^*-\btheta(k))\right],
\end{align*}
which gives
\begin{align}
&2\Eset\left[(\btheta(k)-\theta^*)^T\bar{h}(k)\right]\notag\\ 
&= -\Eset\left[\|\btheta(k)-\theta^*\|_{\Abf}^2\right]-\Eset\left[\|\theta_u(k)-\theta^*\|_{\Abf}^2\right]\notag\\
&\qquad + \Eset\left[(\btheta(k)-\theta_u(k))^T\Abf(\theta^*-\theta_u(k))\right]\notag\\ 
&\qquad +\Eset\left[(\btheta(k)-\theta_u(k))^T\Abf(\theta^*-\btheta(k))\right]\notag\\
&\leq -\sigma_{\min}\Eset\left[\|\btheta(k)-\theta^*\|^2\right]-\sigma_{\min}\Eset\left[\|\theta_u(k)-\theta^*\|^2\right]\notag\\
&\qquad + 2R_0\sigma_{\max}\Eset\left[\|\btheta(k)-\theta_u(k)\|\right]. \label{thm_opt_theta:Eq1a}
\end{align}
Using Eqs.\ \eqref{thm_opt_theta:Eq1a} into Eq.\ \eqref{lem_opt:Ineq} gives
\begin{align}
&\Eset\left[\|\btheta(k+1)-\theta^*\|^2\right]\notag\\ 
&\leq \Eset\left[\|\btheta(k)-\theta^*\|^2\right]  +2\alpha(k)\Eset\left[(\btheta(k)-\theta^*)^T\bar{h}(k)\right]\notag\\
&\qquad  + \frac{4L^2}{N}\alpha^2(k)+\frac{2L}{N}\Eset\left[\alpha(k)\left\|\Qbf\Theta(k)\right\|\right]\notag\\  
&\leq \left(1-\sigma_{\min}\alpha(k)\right)\Eset\left[\|\btheta(k)-\theta^*\|^2\right]+ \frac{4L^2}{N}\alpha^2(k) \notag\\
&\qquad+ \frac{2(L+N\sigma_{\max}R_0)}{N}\Eset\left[\alpha(k)\left\|\Qbf\Theta(k)\right\|\right] \notag\\
&\qquad -\sigma_{\min}\alpha(k)\Eset\left[\|\theta_u(k)-\theta^*\|^2\right].\label{thm_opt_theta:Eq1}
\end{align}
We now consider two choices of stepsizes $\alpha(k)$ with $\beta_2,\beta_3$ given in Eq.\ \eqref{thm_opt_theta:beta1} as follows.

$1.$ Let $\alpha(k) = \alpha\in(0,1\,/\,\sigma_{\min})$ and recall that $\rho = \max\{1-\sigma_{\min}\alpha,\;\delta\}\in(0,1)$. In addition, Eq.\ \eqref{lem_consensus:MainIneq} yields
\begin{align*}
\|\Qbf\Theta(k)\| \leq \frac{\|\Theta(0)\|}{\eta} + \frac{2L\alpha}{\eta(1-\delta)}\cdot 
\end{align*}
Thus, recursively updating Eq.\ \eqref{thm_opt_theta:Eq1}, dropping the negative term, and using the preceding relation yield
\begin{align}
 &\Eset\left[\|\btheta(k+1)-\theta^*\|^2\right]\notag\\ 
&\leq \rho^{k+1}\Eset\left[\|\btheta(0)-\theta^*\|^2\right]+ \frac{4L^2}{N}\alpha^2\sum_{t=0}^{k}\rho^{k-t} \notag\\
&\qquad+ \frac{2(L+N\sigma_{\max}R_0)}{N\eta}\Eset\left[\|\Theta(0)\|\right]\alpha\sum_{t=0}^{k}\rho^{k-t} \notag\\
&\qquad + \frac{L+N\sigma_{\max}R_0}{N\eta}\frac{4L\alpha^2}{1-\delta}\sum_{t=0}^{k}\rho^{k-t}\notag\\
 &\leq \rho^{k+1}\Eset\left[\|\btheta(0)-\theta^*\|^2\right]\notag\\
&\qquad + \frac{2(L+N\sigma_{\max}R_0)}{N\eta}\frac{2\Eset\left[\|\Theta(0)\|\right]\alpha}{1-\rho}\notag\\
&\qquad + \frac{4L(2L+N\sigma_{\max}R_0)}{N\eta(1-\delta)}\frac{\alpha^2}{1-\rho}\cdot
\label{thm_opt_theta:Eq2a}
\end{align}
On the other hand, Eq.\ \eqref{lem_consensus:MainIneq} yields 
\begin{align}
\Eset\left[\|\Qbf\Theta(k)\|^2\right] \leq 2\Eset\left[\|\Theta(0)\|^2\right]\delta^{2k} + \frac{8L^2\alpha^2}{(1-\delta)^2}\cdot\label{thm_opt_theta:Eq2b}
\end{align}
Thus, we obtain Eq.\ \eqref{thm_opt_theta:Ineq1} by using Eqs.\ \eqref{thm_opt_theta:Eq2a} and \eqref{thm_opt_theta:Eq2b}, and $$\Eset[\|\theta_u(k)-\theta^*\|^2]\leq 2\Eset[\|\btheta(k)-\theta^*\|^2]+2\Eset[\|\theta_u(k)-\btheta(k)\|^2].$$ 


$2.$ Let $\alpha(k) = \alpha_0\,/\,(k+1)$ where $\alpha_0 > 1\,/\,\sigma_{\min}$, implying
$1 - \sigma_{\min}\alpha(k) \leq k/(k+1).$ Thus, Eq.\ \eqref{thm_opt_theta:Eq1} gives
\begin{align}
&\Eset\left[\|\btheta(k+1)-\theta^*\|^2\right]\notag\\
&\leq \frac{k}{k+1}\Eset\left[\|\btheta(k)-\theta^*\|^2\right] + \frac{4L^2\alpha_0^2}{N}\frac{1}{(k+1)^2}\notag\\
&\qquad+ \frac{2\alpha_0(L+N\sigma_{\max}R_0)}{N}\frac{\Eset\left[\left\|\Qbf\Theta(k)\right\|\right]}{k+1} \notag\\
&\qquad -\alpha(0)\sigma_{\min}\frac{\Eset\left[\|\theta_u(k)-\theta^*\|^2\right]}{k+1}\notag\\
&\leq \frac{4L^2\alpha_0^2}{N}\sum_{t=0}^k\frac{1}{(t+1)}\frac{1}{k+1}\notag\\
&\qquad + \frac{2\alpha_0(L+N\sigma_{\max}R_0)}{N}\frac{\sum_{t=0}^k\Eset\left[\left\|\Qbf\Theta(t)\right\|\right]}{k+1} \notag\\
&\qquad -\alpha_0\sigma_{\min}\frac{\sum_{t=0}^k\Eset\left[\|\theta_u(k)-\theta^*\|^2\right]}{k+1}.\label{thm_opt_theta:Eq3a}
\end{align}
The integral test gives
$\sum_{t=0}^k1\,/\,(k+1)\leq 1+\ln(k+1).$ In addition, Eq.\ \eqref{lem_consensus:MainIneq} yields\vspace{-0.3cm} 
\begin{align*}
&\sum_{t=0}^k\Eset\left[\|\Qbf\Theta(t)\|\right]\notag\\ 
&\leq \frac{\Eset\left[\|\Qbf\Theta(0)\|\right]}{\eta}\sum_{t=0}^{k}\delta^t +  \frac{2L}{\eta}\sum_{t=0}^{k}\sum_{\ell=0}^{t}\delta^{t-1-\ell}\alpha(\ell)\notag\\
&\leq \frac{\Eset\left[\|\Qbf\Theta(0)\|\right]}{\eta(1-\delta)} + \frac{2L}{\eta}\sum_{\ell=0}^{k}\alpha(\ell)\sum_{t=\ell+1}^{k}\delta^{t}\notag\\
&\leq \frac{\Eset\left[\|\Qbf\Theta(0)\|\right]}{\eta(1-\delta)} + \frac{2L\alpha_0(1+\ln(k+1))}{\eta(1-\delta)}\cdot
\end{align*}
Thus, using the preceding relation into Eq.\ \eqref{thm_opt_theta:Eq3a}, rearranging the terms, and using the Jensen's inequality gives Eq.\ \eqref{thm_opt_theta:Ineq2}.



\section{Conclusion and Discussion}
In this paper, we consider a distributed consensus-based variant of the popular {\sf TD}$(0)$ algorithm for estimating the value function of a given stationary policy. Our main contribution is to provide a finite-time analysis for the performance of distributed {\sf TD}(0), which has not been addressed in the existing literature of {\sf MARL}. In particular, our results mirror what we would expect from using distributed {\sf SGD} for solving static convex optimization problems. A few interesting questions left from this work are the finite-time analysis for the general distributed {\sf TD}$(\lambda)$ and when the policy is not stationary, e.g., distributed actor-critic methods. We believe that this paper establishes fundamental results that enable one to tackle these problems, which we leave for our future research.   

 \section*{Acknowledgements}
This work is partially supported by ARL DCIST CRA W911NF-17-2-0181 and the SRC JUMP program.

\bibliography{refs}
\bibliographystyle{icml2019}

\newpage
\appendix

\section{Supplementary Documents}

Here, we provide the details of analysis of Lemmas \ref{lem:consensus} and \ref{lem:opt}. We first analyze the impact of projection error in Eq.\ \eqref{analysis:theta_k}, where we  will utilize the following result studied in \cite{NedicOP2010}.
\begin{lem}\label{apx_lem:projection_Nedic}
Let $\Xcal$ be a nonempty closed convex set in $\Rset^K$. Then, we have for any $x\in\mathbb{R}^K$ and for all $y\in\Xcal$
\begin{enumerate}[leftmargin = 6mm]
\item[(a)] $(\Pcal_{\Xcal}[x]-x)^T(x-y)\leq -\|\Pcal_{\Xcal}[x]-x\|^2$.
\item[(b)] $\|\Pcal_{\Xcal}[x]-y\|^2\leq \|x-y\|^2 -\|\Pcal_{\Xcal}[x]-x\|^2$.
\end{enumerate}
\end{lem}
Recall that 
\begin{align*}
&\Abf = \Eset_{\pi}\left[\phi(s)(\phi(s)-\gamma\phi(s'))^T\right],\quad b_v = \Eset_{\pi}[r_v\phi(s)]\notag\\
&h_v(k) = b_v -  \Abf\theta_v(k)\\
&M_v(k) = d_v(k)\phi(s(k)) - [b_v - \Abf\theta_v(k)] 
\end{align*}
We now characterize the impact of projection error $e_v$ in Eq.\ \eqref{analysis:theta_k} as given in the following lemma.
\begin{lem}\label{apx_lem:projection}
Suppose that Assumptions \ref{assump:doub_stoch} and \ref{assump:reward}--\ref{assump:Xset} hold. Let $\theta_v(k)$, for all $v\in\Vcal$, be generated by Algorithm \ref{alg:DTD0}. In addition, let $\{\alpha(k)\}$ be a nonnegative nonincreasing sequence of stepsizes. Then, there exists a constant $L_v>0$, for all $v\in\Vcal$, such that 
\begin{align}
\|e_v(k)\| \leq L_v\alpha(k).\label{apx_lem_projection:ev_bound} 
\end{align}\vspace{-0.2cm} 
In addition, let $L = \sum_{v\in\Vcal}L_v$. Then we have
\begin{align}
&\hspace{-0.5cm}-\left(\btheta(k)-\theta^*+\alpha(k)\bar{h}(k)\right)^T\ebar(k)\notag\\
&\hspace{-0.5cm}\quad\leq \frac{L}{N}\alpha(k)\left\|\Qbf\Theta(k)\right\| + \frac{2L^2}{N}\alpha^2(k) - \frac{1}{N}\|\Ebf(k)\|^2.\label{apx_lem_projection:opt}
\end{align}
\end{lem}
\begin{proof}
We first show Eq.\ \eqref{apx_lem_projection:ev_bound}. Indeed, by Assumptions \ref{assump:reward}--\ref{assump:Xset} there exists a positive constant $L_v$, for all $v\in\Vcal$, such that 
\begin{align}
\|h_v(k)\| + \|M_v(k+1)\| \leq L_v.\label{apx_lem_projection:Eq1a}
\end{align}
Moreover, since $\theta_v(k)\in\Xcal$, for all $v\in\Vcal$, $y_v(k)$ in Eq.\ \eqref{alg:theta_v} is in $\Xcal$ since $\Wbf(k)$ is doubly stochastic. Thus, using Lemma \ref{apx_lem:projection_Nedic}(b) with $y = y_v(k)\in\Xcal$ yields Eq.\ \eqref{apx_lem_projection:ev_bound}
\begin{align*}
\|e_v(k)\|^2 &= \|\tilde{\theta}_v(k)-[\tilde{\theta}_v(k)]_{\Xcal}\|^2\notag\\
&\leq \left\|\tilde{\theta}_v(k) - y_v(k)\right\|^2\notag\\
&\stackrel{\eqref{analysis:theta_k}}{=} \left\|\alpha(k)(h_v(k)+M_v(k+1))\right\|^2\notag\\
&\stackrel{\eqref{apx_lem_projection:Eq1a}}{\leq} L_v^2\alpha^2(k),
\end{align*}
which by letting $L=\sum_{v\in\Vcal}L_v$ and using the Cauchy-Schwarz inequality also implies
\begin{align}
\|\Ebf(k)\|^2 = \sum_{v\in\Vcal}\|e_v(k)\|^2 \leq \alpha^2(k)L^2.\label{apx_lem_projection:Eq1b}
\end{align}
We now show Eq.\ \eqref{apx_lem_projection:opt}. Indeed, by Lemma \ref{apx_lem:projection_Nedic}(a) we have
\begin{align}
&-\left(\btheta(k)-\theta^*+\alpha(k)\bar{h}(k)\right)^T\ebar(k)\notag\\
&= -\frac{1}{N}\sum_{v\in\Vcal}(\btheta(k)-\theta^* +\alpha(k)\bar{h}(k))^Te_v(k)\notag\\
&= -\frac{1}{N}\sum_{v\in\Vcal}(\btheta(k)+\alpha(k)\bar{h}(k)-\tilde{\theta}_v(k))^Te_v(k)\notag\\
&\qquad -\frac{1}{N}\sum_{v\in\Vcal}(\tilde{\theta}_v(k)-\theta^*)^Te_v(k)\notag\\
&\leq \frac{1}{N}\sum_{v\in\Vcal}\left\|\btheta(k)-\tilde{\theta}_v(k)\right\|\|e_v(k)\|\notag\\
&\qquad + \frac{1}{N}\sum_{v\in\Vcal}\left\|\alpha(k)\bar{h}(k)\right\|\|e_v(k)\|\notag\\
&\qquad -\frac{1}{N}\sum_{v\in\Vcal}(\tilde{\theta}_v(k)-\theta^*)^T\left(\tilde{\theta}_v(k)-\left[\tilde{\theta}_v(k)\right]_{\Xcal}\right)\notag\\
&\leq \frac{1}{N}\sum_{v\in\Vcal}\left\|\btheta(k)-\tilde{\theta}_v(k)\right\|\|e_v(k)\|\notag\\
&\qquad + \frac{1}{N}\sum_{v\in\Vcal}\left\|\alpha(k)\bar{h}(k)\right\|\|e_v(k)\|\notag\\
&\qquad -\frac{1}{N}\|\Ebf(k)\|^2,\label{apx_lem_projection:Eq2}
\end{align}
where the last inequality is due to Lemma \ref{apx_lem:projection_Nedic}(a). First, using Eqs.\ \eqref{analysis:theta_k}, \eqref{apx_lem_projection:ev_bound}, and \eqref{apx_lem_projection:Eq1a} gives
\begin{align*}
&\sum_{v\in\Vcal}\left\|\btheta(k)-\tilde{\theta}_v(k)\right\|\|e_v(k)\|\notag\\
&\leq \alpha(k)\sum_{v\in\Vcal}L_v\left\|\btheta(k)-\tilde{\theta}_v(k)\right\|\notag\\
&\leq \alpha(k)\sum_{v\in\Vcal}L_v\left\|\btheta(k)-\sum_{u\in\Ncal_{v}(k)}W_{vu}(k)\theta_{u}(k)\right\|\notag\\
&\qquad + \alpha(k)\sum_{v\in\Vcal}L_v\left\| \alpha(k)(h_v(k)+M_{v}(k))\right\| \notag\\
&\leq L\alpha(k)\|\Qbf\Theta(k)\| + L^2\alpha^2(k),
\end{align*}
where the last inequality is due to the Cauchy-Schwarz inequality and $\Wbf$ is doubly stochastic. Second, using Eqs.\ \eqref{apx_lem_projection:ev_bound} and \eqref{apx_lem_projection:Eq1a} yields
\begin{align*}
&\sum_{v\in\Vcal}\left\|\alpha(k)\bar{h}(k)\right\|\|e_v(k)\|\leq L^2\alpha^2(k).
\end{align*}
Thus, using the preceeding relations into Eq.\ \eqref{apx_lem_projection:Eq2} immediately gives Eq.\ \eqref{apx_lem_projection:opt}.
\end{proof}

\subsection{Proof of Lemma \ref{lem:consensus}}

\begin{lem}\label{apx_lem:consensus}
Suppose that Assumptions \ref{assump:BConnectivity}--\ref{assump:Xset} hold. Let $\theta_v(k)$, for all $v\in\Vcal$, be generated by Algorithm \ref{alg:DTD0}. In addition, let $\{\alpha(k)\}$ be a nonnegative nonincreasing sequence of stepsizes. Then for some positive constant $L$ we have
\begin{enumerate}[leftmargin = 5mm]
\item The consensus error $\Qbf\Theta(k)$ satisfies
\begin{align}
\hspace{-0.8cm}\|\Qbf\Theta(k)\| \leq \frac{1}{\eta}\delta^{k }\|\Theta(0)\| + \frac{2L}{\eta}\sum_{t=0}^{k-1}\delta^{(k-1-t)}\alpha(t).\label{apx_lem_consensus:MainIneq}   
\end{align}
\item In addition, we obtain 
\begin{align}
\hspace{-0.5cm}\sum_{t=0}^{k}\alpha(t)\|\Qbf\Theta(t)\| \leq \frac{\alpha(0)\|\Theta(0)\|}{\eta(1-\delta)} + \frac{2L}{\eta(1-\delta)}\sum_{t=0}^k\alpha^2(t).    \label{apx_lem_consensus:FiniteBound}
\end{align}
\end{enumerate}
\end{lem}

\begin{proof}
We first show Eq. \eqref{apx_lem_consensus:MainIneq}. Since $\Wbf(k)$ are doubly stochastic we have 
\begin{align*}
    \Wbf(k)\Theta(k) - \1\btheta(k)^T = \Wbf(k)\Big( \Theta(k) - \1\btheta(k)^T\Big). 
\end{align*}
By Eqs.\ \eqref{analysis:Theta_k} and \eqref{analysis:theta_bar} we consider
\begin{align*}
&\Qbf\Theta(k+1) = \Theta(k+1) - \1\btheta(k+1)^T\notag\\
&= \Wbf(k)\Theta(k) + \alpha(k)\Big[\Hbf(k) + \Mbf(k) \Big] - \Ebf(k)\notag\\
&\qquad - \Big[\1\btheta(k)^T +  \alpha(k)\1\Big[\bar{h}(k) + \bar{M}(k) \Big]^T + \1\ebar(k)^T\Big]\notag\\
&= \Wbf(k) \Big( \Theta(k) - \1\btheta(k)^T\Big) + \alpha(k)\Big(\Hbf(k) - \1\bar{h}(k)^T\Big)\notag\\
&\qquad + \alpha(k)\Big(\Mbf(k) - \1\bar{M}(k)^T\Big)- (\Ebf(k)-\1\ebar(k)^T)\notag\\
&= \Wbf(k)\Qbf\Theta(k) + \alpha(k)\Qbf\Big(\Hbf(k) + \Mbf(k)\Big) - \Qbf\Ebf(k)\notag\\
&= \prod_{t=0}^{k}\Wbf(t)\Qbf\Theta(0)- \sum_{t=0}^{k}\prod_{\ell = t+1}^k \Wbf(\ell)\Qbf\Ebf(t)\notag\\ 
&\qquad + \sum_{t=0}^{k}\alpha(t)\prod_{\ell = t+1}^k \Wbf(\ell)\Qbf\Big(\Hbf(t) + \Mbf(t)\Big),
\end{align*}
where in the third equality we use the definition of $\Qbf = \Ibf - 1/N\,\1\1^T$. Taking the Frobenius norm on both sides of the equation above and using the triangle inequality give
\begin{align}
&\|\Qbf\Theta(k+1)\|\notag\\ 
&\leq \left\|\prod_{t=0}^{k}\Wbf(t)\Theta(0)\right\|+ \sum_{t=0}^{k}\left\|\prod_{\ell = t+1}^k \Wbf(\ell)\Qbf\Ebf(t)\right\|\notag\\ 
&\qquad + \sum_{t=0}^{k}\alpha(t)\left\|\prod_{\ell = t+1}^k \Wbf(\ell)\Qbf\Big(\Hbf(t) + \Mbf(t)\Big)\right\|\notag\\
&\stackrel{\eqref{const:Bmix}}{\leq}\eta^{\lfloor (k+1) / \Bcal \rfloor}\|\Theta(0)\| + \sum_{t=0}^{k}\eta^{\lfloor (k-t)/\Bcal \rfloor }\|\Ebf(t)\|\notag\\ 
&\qquad + \sum_{t=0}^{k}\eta^{\lfloor (k-t) / \Bcal\rfloor }\alpha(t)\|\Hbf(t)+\Mbf(t)\|\notag\\
&\leq\frac{1}{\eta}\delta^{(k+1)}\|\Theta(0)\| + \frac{1}{\eta}\sum_{t=0}^{k}\delta^{(k-t)}\|\Ebf(t)\|\notag\\ 
&\qquad + \frac{1}{\eta}\sum_{t=0}^{k}\delta^{ (k-t) }\alpha(t)\|\Hbf(t)+\Mbf(t)\|,
\label{apx_lem_consensus:Eq1}
\end{align}
where in the second inequality we also use the spectral of $\Qbf$ is less than $1$. The last inequality follows from noting that $\eta^{\lfloor (k+1) / \Bcal \rfloor} \leq \eta^{(k+1) / \Bcal -1} \leq \delta^{(k+1)}/\eta$.
Thus, using Eqs.\ \eqref{apx_lem_projection:Eq1a} and \eqref{apx_lem_projection:Eq1b} into Eq.\ \eqref{apx_lem_consensus:Eq1} yields Eq.\ \eqref{apx_lem_consensus:MainIneq}. 

Second, multiply both sides of Eq.\ \eqref{apx_lem_consensus:MainIneq} by $\alpha(k)$ we obtain
\begin{align*}
&\alpha(k)\|\Qbf\Theta(k)\|\notag\\
&\qquad\leq \frac{\|\Theta(0)\|\alpha(k)}{\eta}\delta^{ k} + \frac{2L\alpha(k)}{\eta}\sum_{t=0}^{k-1}\delta^{(k-1-t)}\alpha(t)\notag\\ 
&\qquad\leq \frac{\alpha(0)\|\Theta(0)\|}{\eta}\delta^{ k}+ \frac{2L}{\eta}\sum_{t=0}^{k-1}\delta^{(k-1-t)}\alpha^2(t),
\end{align*}
where the last inequality is due to the nonincreasing of $\alpha(k)$. Summing up both sides of the preceding equation over $k$ from $0$ to $K$ for some $K\geq0$ we obtain Eq.\ \eqref{apx_lem_consensus:FiniteBound}, i.e.,
\begin{align*}
&\sum_{k=0}^{K}\alpha(k)\|\Qbf\Theta(k)\|\notag\\
&\leq  \frac{\alpha(0)\|\Theta(0)\|}{\eta}\sum_{k=0}^{K}\delta^{k} + \frac{2L}{\eta}\sum_{k=0}^{K}\sum_{t=0}^{k-1}\delta^{ (k-1-t)}\alpha^2(t)\notag\\
&\leq \frac{\alpha(0)\|\Theta(0)\|}{\eta(1-\delta)} + \frac{2L}{\eta}\sum_{t=0}^{K-1}\alpha^2(t)\sum_{k=t+1}^{K}\delta^{k}\notag\\
&\leq \frac{\alpha(0)\|\Theta(0)\|}{\eta(1-\delta)} + \frac{2L}{\eta(1-\delta)}\sum_{t=0}^K\alpha^2(t).
\end{align*}
\end{proof}
\begin{remark}
Note that if $\lim_{k\rightarrow\infty}\alpha(k) = 0$, Eq.\ \eqref{apx_lem_consensus:MainIneq} also implies that the agents achieve a consensus on their estimates. Indeed, using Eq.\ \eqref{apx_lem_consensus:MainIneq} while ignoring the factor $\delta^{1/\Bcal}$
\begin{align*}
&\|\Qbf\Theta(k+1)\| \leq  \delta^{k+1}\|\Theta(0)\| + 2L\sum_{t=0}^{k}\delta^{k-t}\alpha(t)\notag\\
&= \delta^{k+1}\|\Theta(0)\|\notag\\ 
&\qquad + 2L\left[\sum_{t=0}^{\lfloor k/2 \rfloor}\delta^{k-t}\alpha(t) + \sum_{t=\lceil k/2 \rceil}^{k}\delta^{k-t}\alpha(t) \right]\notag\\
&\leq \delta^{k+1}\|\Theta(0)\|\notag\\ 
&\qquad + 2L\left[\alpha(0)\sum_{t=0}^{\lfloor k/2 \rfloor}\delta^{k-t} +\alpha(\lceil k/2 \rceil) \sum_{t=\lceil k/2 \rceil}^{k}\delta^{k-t}\right]\notag\\
&\leq \delta^{k+1}\|\Theta(0)\| + 2L\left[\frac{\alpha(0)}{1-\delta}\delta^{\lceil k/2 \rceil} +\frac{\alpha(\lceil k/2 \rceil)}{1-\delta}\right],
\end{align*}
which implies that $\lim_{k\rightarrow\infty}\|\Qbf\Theta(k)\| = 0$. In the preceding equation, we use the nonnegativity and nonincreasing property of $\alpha(k)$ in the first inequality. 
\end{remark}

\subsection{Proof of Lemma \ref{lem:opt}}
We now proceed to show Lemma \ref{lem:opt} as restated as follows.  
\begin{lem}\label{apx_lem:opt}
Suppose that Assumptions \ref{assump:BConnectivity}--\ref{assump:Xset} hold. Let $\theta_v(k)$, for all $v\in\Vcal$, be generated by Algorithm \ref{alg:DTD0}. In addition, let $\{\alpha(k)\}$ be a nonnegative nonincreasing sequence of stepsizes. Then there exists a positive constant $L$ such that
\begin{align}
&\Eset\left[\|\btheta(k+1)-\theta^*\|^2\right]\notag\\ 
&\;\leq \Eset\left[\|\btheta(k)-\theta^*\|^2\right]+  \frac{2L}{N}\alpha(k)\Eset\left[\left\|\Qbf\Theta(k)\right\|\right]\notag\\
&\qquad + \frac{4L^2\alpha^2(k)}{N}+ 2\alpha(k)\Eset\left[(\btheta(k)-\theta^*)^T\bar{h}(k)\right].\label{apx_lem_opt:Ineq} 
\end{align}
\end{lem}

\begin{proof}
Recall that $\Abf\theta^* = 1/N\left(\sum_{v\in\Vcal}b_v\right)$ for a given fixed point $\theta^*$ of the projected Bellman equation. In addition, let $L = \sum_{v\in\Vcal}L_v$, where $L_v$ is given in Eq.\ \eqref{apx_lem_projection:Eq1a}. 
First, Eq.\ \eqref{analysis:theta_bar} yields
\begin{align}
&\|\btheta(k+1)-\theta^*\|^2 \notag\\ 
&= \|\btheta(k)-\theta^* + \alpha(k)\bar{h}(k) + \alpha(k)\bar{m}(k) - \ebar(k)\|^2\notag\\ 
&= \|\btheta(k)-\theta^*\|^2 + \left\|\alpha(k)\Big(\bar{h}(k) +\bar{m}(k)\Big)\right\|^2 + \|\ebar(k)\|^2 \notag\\
&\qquad + 2\alpha(k)(\btheta(k)-\theta^*)^T\Big(\bar{h}(k) +\bar{m}(k)\Big) \notag\\
&\qquad -2(\btheta(k)-\theta^*)^T\ebar(k) - 2\alpha(k)\Big(\bar{h}(k) +\bar{m}(k)\Big)^T\ebar(k).\label{apx_lem:Eq1}
\end{align}
Since the data are sampled i.i.d from $\pi$ and under Assumptions \ref{assump:reward}-- \ref{assump:Xset} we have $\{M_v(k)\}$ in Eq.\ \eqref{analysis:hM} is the sequence of i.i.d Martingale difference noise, i.e.,
\begin{align*}
\Eset[M_v(k)\,|\,\Fcal_{k}] = 0,
\end{align*}
where $\Fcal_{k}$ is the filtration contains all the variables $\theta_v(k)$ generated by Algorithm \ref{alg:DTD0} up to time $k$, i.e., $\Fcal_{k}$ is the $\sigma$-algebra generated by $$\{\Theta(0)\ldots,\Theta(k)\}.$$ Thus, taking the conditional expectation on both sides of Eq.\ \eqref{apx_lem:Eq1} w.r.t $\Fcal_k$ we obtain
\begin{align*}
&\Eset\left[\|\btheta(k+1)-\theta^*\|^2\,|\,\Fcal_k\right]\notag\\ 
&= \|\btheta(k)-\theta^*\|^2 + \left\|\alpha(k)\Big(\bar{h}(k) +\bar{m}(k)\Big)\right\|^2 + \|\ebar(k)\|^2 \notag\\
&\qquad + 2\alpha(k)(\btheta(k)-\theta^*)^T\bar{h}(k)\notag\\ 
&\qquad -2\Big(\btheta(k)-\theta^*+\alpha(k)\bar{h}(k)\Big)^T\ebar(k),
\end{align*}
which by taking the expectation on both sides, and using Eqs.\ \eqref{apx_lem_projection:ev_bound}--\eqref{apx_lem_projection:Eq1a} and the Cauchy-Schwarz inequality implies Eq.\ \eqref{apx_lem_opt:Ineq}
\begin{align*}
&\Eset\left[\|\btheta(k+1)-\theta^*\|^2\right]\notag\\ 
&\leq \|\btheta(k)-\theta^*\|^2 + \frac{4L^2}{N}\alpha^2(k)+ \frac{2L}{N}\alpha(k)\left\|\Qbf\Theta(k)\right\| \notag\\
&\qquad  + 2\alpha(k)(\btheta(k)-\theta^*)^T\bar{h}(k). 
\end{align*}
\end{proof}

\end{document}